\documentclass[11pt]{article}
\advance\topmargin by -\headheight
\advance\topmargin by -\headsep
\evensidemargin=\oddsidemargin
\usepackage{amsmath,amsfonts,amssymb,amsthm}
\usepackage{framed}
\usepackage[dvipsnames]{color}
\usepackage[T1]{fontenc}
\usepackage{times} 
\usepackage[cp1250]{inputenc}  
\makeatother
\newtheorem{thm}{Theorem}[section]

\newtheorem{lemma}[thm]{Lemma}

\theoremstyle{definition}

\numberwithin{equation}{section}
\newcommand{\CA}{\mathcal{A}} 
\renewcommand{\a}{\alpha} 

\newcommand{\C}{\mathbb{C}} 

\renewcommand{\H}{\mathcal{H}} 
\DeclareMathOperator{\id}{id} 
\newcommand{\T}{\mathbb{T}} 
\newcommand{\oh}{{\tfrac{1}{2}}} 
\newcommand{\R}{\mathbb{R}} 
\newcommand{\ohh}{{\scriptstyle\frac{1}{4}}} 
\newcommand{\sq}{\unskip\nobreak\kern5pt\nobreak\vrule height4pt width4pt depth0pt} 
\newcommand{\Z}{\mathbb{Z}} 

\def\<#1,#2>{\langle#1,#2\rangle} 

\def\CH{\mathcal{H}}

\def\suq2{{\cal A}(SU_q(2))}
\definecolor{lgray}{rgb}{0.2,0.2,0.2}
\definecolor{shadecolor}{named}{GreenYellow}
\newbox\ncintdbox \newbox\ncinttbox
\setbox0=\hbox{$-$} \setbox2=\hbox{$\displaystyle\int$}
\setbox\ncintdbox=\hbox{\rlap{\hbox
    to \wd2{\kern-.1em\box2\relax\hfil}}\box0\kern.1em}
\setbox0=\hbox{$\vcenter{\hrule width 4pt}$}
\setbox2=\hbox{$\textstyle\int$}
\setbox\ncinttbox=\hbox{\rlap{\hbox
    to \wd2{\kern-.14em\box2\relax\hfil}}\box0\kern.1em}

\title{Curved noncommutative torus and Gauss--Bonnet}
\author{Ludwik D\k{a}browski${}^1$ and Andrzej Sitarz${}^{2,3,*}$ \\ 
\vbox{
\small
\begin{center}
${}^1$ SISSA (Scuola Internazionale Superiore di Studi Avanzati), \\
via Bonomea 265, 34136 Trieste, Italy \\ \ \\
${}^2$ Institute of Physics, Jagiellonian University,\\
Reymonta 4, 30-059 Krak\'ow, Poland \\ \ \\
${}^3$ Institute of Mathematics of the
Polish Academy of Sciences,\\
ul. Sniadeckich 8, Warszawa, 00-950 Poland \\ \ \\
${}^*$ Partially supported by MNII grant 189/6.PRUE/2007/7.\\
\end{center}
}
}
\date{}
\begin{document}
\maketitle
\begin{abstract}
We study perturbations of the flat geometry of the noncommutative 
two-dimensional torus $\T^2_\theta$ (with irrational $\theta$). They 
are described by spectral triples $(A_\theta, \H, D)$, with the 
Dirac operator $D$, which is a differential operator with coefficients 
in the commutant of the (smooth) algebra $A_\theta $ of $\T_\theta$. 
We show, up to the second order in perturbation, that the $\zeta$-function 
at $0$ vanishes and so the Gauss-Bonnet theorem holds. We also 
calculate first two terms of the perturbative expansion of the 
corresponding  local scalar curvature.
\end{abstract}
\thispagestyle{empty}
\section{Introduction}
Starting with the seminal paper \cite{CoTr} (see also \cite{cohcon}) there has been a growing 
interest in study of the metric and curvature issues on the most popular example in 
the realm of noncommutative geometry: the noncommutative two-torus $\T_\theta$
\cite{fatkha}, \cite{bhumar}, \cite{conmos2}, \cite{fatkha1}.
This has been pursued in the framework of spectral triples \cite{Co94, Co95},
mainly with their twisted (or modular) version, with the metric obtained from 
the flat one by a conformal rescaling factor. 

Building on the analysis of 
spectral triples on the principal $U(1)$ bundles \cite{DS}, \cite{DSZ} we propose 
another, more general class of curved perturbations of  the flat geometry.
They are described by spectral triples $(A_\theta, \H, D)$, with the Dirac operator 
$D$, which is a differential operator with coefficients in 
$M_2(\C) \otimes (J A_\theta J^*)$, where  $A_\theta $ is the (smooth) algebra 
of $\T_\theta$ and $J$ is the real structure (charge conjugation).
The fact that $JA_\theta J^*$ is in the commutant of $A_\theta $ is here essential 
to guarantee that $D$ has bounded commutators with functions in  $A_\theta $. 

We take the perturbations of the standard equivariant Dirac operator on the noncommutative 
torus, which correspond to arbitrary perturbations of the standard flat metric. 
We compute that the $\zeta$-function at $0$ vanishes up to the second order 
in the perturbation parameter $\varepsilon$.
Thus we show the Gauss-Bonnet formula to hold (in the same approximation) 
for a more general class of Dirac operators than  previously studied,
which strongly indicates that it should hold exactly.
In section \ref{curv} we compute up to the order $\varepsilon^2$ 
the corresponding local scalar curvature.

We also briefly touch upon few other properties as 
the first order differential calculus, the orientation postulate 
(which holds under commutativity of certain components of the orthonormal coframe) 
and absolute continuity.

\section{The curved geometry of torus}
We start with the recollection of the Riemannian geometry of the classical 2-dimensional
torus.

\subsection{Classical torus $\T^2$}

Let $x^\mu \in [0, 2\pi )$, $\mu= 1,2$ be the coordinates on $\T^2=S^1\times S^1$. 
The derivations  $\delta_\mu =i \frac{\partial}{\partial x^\mu}$ generate the 
usual action of $U(1) \times U(1)$ on $\T^2$ and on 
its algebra of smooth functions $A=C^\infty(\T^2,\C)$.
Given a metric tensor $g= g_{\mu\nu}d x^\mu dx^\nu$ on $\T^2$  we introduce and use
a global orthonormal frame (basis) $e=\{e_j\}$, $e_j=\{e_j^\mu \partial_\mu\}$, 
$j=1,2$, of the tangent bundle $T \T^2$. Clearly $e'=\{e_j'\}$ is another orthonormal 
frame for $g$ if and only if  $e'$ differs from $e$ by a (point dependent) rotation.
Note that the $2\times 2$ matrix of smooth real valued functions 
$e_j^\mu$ is nondegenerate at all points of $\T^2$ and that one has
$$
\sum_j e_j^\mu e_j^\nu =  g^{\mu\nu},
$$
where $g^{\mu\nu}$ is the inverse matrix of  $g_{\mu\nu}$. 
The metric gives rise in a natural way to a measure on the torus 
$\mu_g = \sqrt{|g|} \mu_0$, where $|g|=\det (g_{\mu\nu})$

and by $\mu_0$ we denote the standard Lebesgue measure on $\T^2$ 
(corresponding to the flat metric).

It is straightforward to see that the scalar curvature has the following  
expression in terms of the orthonormal frame:
\begin{equation}\label{Ron}
R = 2\, {\cal L}_{e_i}(c_{ijj}) - c_{kii} c_{kjj} - \frac{1}{4} c_{ijk}c_{ijk}
      - \frac{1}{2} c_{ijk} c_{kji}\, , 
\end{equation}
where $c_{ijk}$ are the structure constants of the commutators (of vector fields)
$$[ e_i, e_j ] =   c_{ijk} e_k .$$

To make contact with our computations in the noncommutative case, we present here the 
perturbative expansion of the scalar curvature $R$.
Assuming the small perturbation of the globally flat orthonormal frame:
$$e_i = e_i^\mu \partial_\mu , \qquad e^\mu_i = \delta^\mu_i + \varepsilon h_{i}^\mu, $$
where $h_{ij}$ is a matrix of arbitrary smooth real functions on the manifold,
we obtain (up to $\varepsilon^2$)
\begin{equation}\label{curveps}
\begin{aligned}
R\;& = 2 \varepsilon \left( h^i_{i,jj} - h^j_{i,ij}\right) 
     + 2 \varepsilon^2 h_{i}^j \left( 2h_{k,ij}^k - h_{i,jk}^k 
      - h_{j,kk}^i \right) \\ 
   & + \frac{1}{2} \varepsilon^2 h_{i,k}^j \left( 
       h_{k,i}^j - h_{i,k}^j +3 h_{j,i}^k - h_{i,j}^k 
     - 5 h_{j,k}^i \right) \\
   & + \varepsilon^2 h_{i,j}^i \left( 
       2 h_{j,k}^k + 2 h_{k,k}^j  - h_{k,j}^k \right) 
    - \varepsilon^2 h_{i,j}^j h_{i,k}^k.
\end{aligned}
\end{equation}
We shall also need the perturbative expansion of $\sqrt{|g|}\, R$,
which since 
$$ \sqrt{|g|} \sim 1 - \varepsilon (h_k^k),$$
 has just the terms of order $\varepsilon^2$
\begin{equation}\label{differs}
 - 2 \varepsilon^2 h_k^k \left( (h_i^i)_{jj} - (h_i^j)_{ij} \right),
\end{equation}
in addition to the terms in the expansion (\ref{curveps}) of the scalar curvature $R$ alone.

On $\T^2$ there are four spin structures but for any of them we can take  the associated 
spinor bundle $\Sigma \T^2$ as the trivial rank $2$ 
hermitian 
vector bundle on $\T^2$. A $U(1)$ action either lifts directly to a spin 
structure and then to an action
$$ \kappa: U(1)\times \Sigma \T^2 \to \Sigma \T^2,$$ and so it lifts to the action 
on sections $C^\infty(\T^2,\Sigma \T^2 )$ of $\Sigma \T^2$, or to a projective action, 
i.e. to the action of a non-trivial double cover of $U(1)$ (which happens to be still $U(1)$ as a group). The generators of the action of $U(1)\times U(1)$, 
that implement (via a commutator) the two derivations $\delta_\mu$ of $A$
will be denoted by the same symbol $\delta_\mu$.  

The standard Dirac operator $D_g$, which comes from the metric compactible and torsion-free
spin connection can be globally expressed on smooth sections of $\Sigma \T^2$ in terms of 
the orthonormal frame as:
\begin{equation}
D_g = \sum\limits^2_{j,\mu =1}  \left( \sigma^j e_j^\mu \delta_\mu  \right) 
 + \frac{1}{2} c_{122  }\ \sigma^1 + \frac{1}{2} c_{211  }\ \sigma^2 \ ,
\end{equation}
where $\sigma^j$ are the hermitian Pauli matrices in  $M_2(\C)$: 
\begin{equation}
\label{clifrel}
\sigma^1\sigma^2 = -\sigma^2\sigma^1= i \sigma^3,\quad\sigma^3 := \hbox{diag}(1,-1).
\end{equation}
$D_g$ extends to an unbounded selfadjoint operator on $L^2(\Sigma \T^2, \mu_g)$. 
To study the spectral properties we shall employ
a convenient isometry of Hilbert spaces:
$$ Q: L^2(\Sigma \T^2, \mu_g) \to L^2(\Sigma \T^2, \mu_0), \quad \quad Q(\psi) = 
 \hbox{det}(g)^\frac{1}{4} \psi, $$
and consider the operator $Q D_g Q^{-1}$ on $L^2(\Sigma \T^2, \mu_0)$. 
Note that the principal symbols of  $Q D_g Q^{-1}$ and $ D_g $ are equal
and only the zero-order terms are different. 
As the zero-order term is not significant for the computations of the curvature we can, 
in fact, restrict ourselves to yet another operator on $\Sigma \T^2$:
\begin{equation}
D = \sum\limits^2_{j,\mu =1}  \left( \sigma^j e_j^\mu \delta_\mu  + \frac{1}{2} 
\sigma^j (\delta_\mu e_j^\mu) \right),
\label{odir}
\end{equation}
which extends to selfadjoint operator on $L^2(\Sigma \T^2, \mu_0)$ and differs from
$\Xi D_g \Xi^{-1}$ only by a zero-order term. 

Such a term in two dimensions consists of a Clifford image of some selfadjoint one-form.
For the purpose of the paper the crucial fact is that 
$(C^\infty(\T^2), L^2(\Sigma T^2, \mu_0), D)$ 
is an even spectral triple, with the grading $\chi := \sigma^3$. 
Though it is not necessarily real, 
it contains a first order, elliptic  Dirac-type operator $D$, 
which is a fluctuaction of the genuine Dirac operator, 
that belongs to a real spectral triple (c.f. \cite{Co95}).

Moreover, the distance defined by $D$,
\begin{equation}
\label{dist}
d_D (x,y) = \sup  \{ | a(x) - a(y) | ~:~ a\in A ,
 ~ \| [D,a]\| \leq 1 \} \ 
\end{equation}
is exactly the geodesic distance of the metric $g$. 
This is so, because only the principal symbol of $D$ has nontrivial commutators with 
$a \in C^\infty(\T^2)$ and the norm is still the supremum norm. 

Clearly, since $D$ is in general different from the canonical Dirac operator, it might 
(a priori) not minimize the Einstein action functional or (in the special case of 
two-dimensional torus) do not satisfy the Gauss-Bonnet theorem, stated in terms of its 
spectral invariant. However, as the perturbations of the Dirac by one-forms do not 
contribute the leading terms of the heat-kernel coefficients we notice that the relevant 
spectral invariants related to curvature shall be unchanged. \\

Consider now two functionals which to any element $a \in C^\infty(\T^2)$ assign the value
of the zeta function at 0
$$ 
\zeta_{a,D}(0) = \hbox{Tr\ } a |D|^{-s}|_{s=0} \quad  \hbox{and} \quad
\zeta_{a,D_g}(0) = \hbox{Tr\ } a |D_g|^{-s}|_{s=0} , $$
where
$\hbox{Tr\ }$ is the trace over the relevant Hilbert space.
Since both these functionals are spectral and depend only on the principal symbol of the operators $D$ and $D_g$ respectively, they are necessarily identical,
$ \zeta_{a,D}(0) = \zeta_{a,D_g}(0)$.
Assuming that both operators have empty kernel, 
the known expression for $D_g$ in terms of scalar curvature,
 $\zeta_{a,D_g}(0) = \frac{1}{12\pi} \int_{\T^2} a R \,\mu_g$
permits us to have a similar  expression for $D$
$$ 
\zeta_{a,D}(0) \frac{1}{12\pi} \int_{\T^2} a \sqrt{|g|} R \,\mu_0\ ,$$
in which however we had to use  the standard Lebesgue measure $\mu_0$
that corresponds to the canonical trace on the algebra $C^\infty(\T^2)$.
This means that from the functional $\zeta_{a,D}(0)$ we retrieve 
the information about the product $\sqrt{|g|}\, R$ rather than about $R$ itself
(in fact it may be not possible to define  $\sqrt{|g|}$ in the noncommutative setup).\\

\subsection{Noncommutative torus $\T^2_\theta $}

The $C^*$-algebra of the $2$-dimensional noncommutative torus $\T_\theta$ is generated 
by two unitary elements $U_i$, $i=1,2$, with the relations 
$$ U_1 U_2 = e^{2\pi i\theta} U_2 U_1,$$ 
where  $0< \theta < 1$ is irrational.
The smooth subalgebra $A_\theta$ consists of all elements of the form
$$ a = \sum_{k,l \in \Z} \alpha_{kl} U_1^k U_2^l , $$
where $\alpha_{k,l}$ is a rapidly decreasing sequence. 

The natural action of $U(1) \times U(1)$ by automorphisms, gives, in its infinitesimal
form, two linearly independent derivations on the algebra $A_\theta$, given on the
generators as:
\begin{equation}
\delta_1 (U_1) = U_1, \delta_1 (U_2) = 0,\;\;
\delta_2 (U_2) = U_2, \delta_2 (U_1) = 0\ .
\label{u1act}
\end{equation}

The canonical trace on $A_\theta$ 
$$ \mathfrak{t}(a) = \a_{00}, $$
is invariant with respect to the action of $U(1) \times U(1)$, and therefore: 
$$ \mathfrak{t}(\delta_j(a)) = 0,\quad \forall j=1,2.$$
It is easy to see that the trace extends uniquely to the $C^*$-algebra
of the noncommutative torus.

Let $\CH_0$ be the Hilbert space of the GNS construction with 
respect to the trace $\mathfrak{t}$ on $\T_\theta$, and $\pi$ the 
associated faithful representation. With the orthonormal basis, 
$\epsilon_{k,l}$, of $\CH_0$ we have:
$$ 
\begin{aligned}
U_1 \epsilon_{k,l} &= \epsilon_{k+1,l}, \\
U_2 \epsilon_{k,l} &= e^{2\pi ik \theta} \epsilon_{k,l+1}, \\
\end{aligned}
$$
where $k,l$ are in $\Z$ or $\Z+\oh$ depending on the choice of the spin structure (see \cite{PaSi}). 
We double the Hilbert space taking $\CH = \CH_0 \otimes \C^2$, 
with the diagonal representation of the algebra. We take $J$, the real structure to be
$$J = i \sigma^2 \circ J_0, $$
where $J_0$ is the canonical Tomita-Takesaki antilinear map on 
the Hilbert space $\CH_0$:
$$ J_0 \epsilon_{k,l} = \epsilon_{-k,-l}. $$
By construction, the conjugation by $J$ maps the algebra $A_\theta$ to its commutant:
$$ [a, Jb^*J^{-1}] = 0, \;\;\; \forall a,b \in A_\theta.$$

The explicit action of the derivations on the basis is given by 
$$ 
\delta_1 \epsilon_{k,l} = k \epsilon_{k,l}, \quad
\delta_2 \epsilon_{k,l} = l \epsilon_{k,l}.
$$

Before we continue let us observe two important facts. The derivations 
$\delta_i$, $i=1,2$ anticommute with $J$ and hence could be extended as 
derivation to $(A_\theta)'$. In fact, since they are realized in the 
representation as commutators, they extend as derivations to the image 
of $A_\theta \otimes (A_\theta)'$.

Moreover, the canonical trace $\mathfrak{t}$ is also invariant 
under conjugation by $J$ in the following sense:
$$ \mathfrak{t}(J a J^{-1}) = \overline{\mathfrak{t}(a)}, $$
and therefore it makes sense to consider its extension to the
image of $A_\theta \otimes (A_\theta)'$ by setting:

$$ \mathfrak{t}(a b^o) = a_{0,0} b_{0,0}, $$
where 
$$ a = \sum_{m,n} a_{m,n} U^m V^m, \quad \quad 
   b = \sum_{m,n} b_{m,n} \tilde{U}^m \tilde{V}^n,$$
and
$$ \tilde{U} = J U^* J^{-1}, \tilde{V} = JV^*J^{-1}.$$

Clearly, the functional is well-defined on the image 
of tensor product of smooth algebras $A_\theta$ and 
$A_\theta^o$. In fact, it a trace obtained from the
,,flat'' Laplacian on the noncommutative torus, 
$\Delta = (\delta_1)^2 + (\delta_2)^2$ through the 
formula:
$$ \mathfrak{t}(a b^o) = 
\hbox{Res}_{s=2} \left( \hbox{Tr\ } \pi(a) \pi(b^0) \Delta^{-s/2} \right). $$
It follows directly from Theorem 2.6 in \cite{MCC} that the trace factorizes.

We take as the Dirac-type operator $D$, the operator of the form:
\begin{equation}
D = \sum_{j,\mu=1}^2 \left( \sigma^j e_{j}^\mu \delta_\mu +
\frac{1}{2} \sigma^j \delta_\mu(e_{j}^\mu) \right),
\label{dirtor} 
\end{equation}
 
where $\sigma^i$ are the usual selfadjoint Pauli matrices and $e^i_j$
are selfadjoint elements of $JA_\theta J^*$ such that the matrix $e^j_i$ is invertible. \\

Such operator is of the general type of generalized differential operators
on the noncommutative torus (or, more precisely, on the module over $A_\theta$), 
which were introduced in studied in \cite{Co80,Co94}:
$$ P = \sum_{\alpha,\beta} C_{\alpha,\beta} \delta_1^\alpha \delta_2^\beta,$$
where $C_{\alpha\beta}$ are $A_\theta$ endomorphisms. Clearly this is the
case of $D$ given by (\ref{dirtor}), as the elements of $JA_\theta J^*$ are 
$A_\theta$-endomorphisms of the the trivial module. 

Notice that from the assumption
of invertibility of  the matrix $e^j_i$ it follows that $\sigma^i e^j_i \xi_j$ is invertible as 
an element of $M_2(JA_\theta J^*)$ for arbitrary $\xi\in \R^2 \setminus \{0\}$.\\

Thus, following \cite{Co94}[Definition 1, p.358]  $D$ could be seen as an elliptic operator.\\

The principal symbol of $D^2$ is of the form 
\begin{equation}
\sum_{i,j} \left(  P_{ij} \xi^i \xi^j \right), 
\label{laplace}
\end{equation} 
where $P_{ij}$ are elements of $M_2(JA_\theta J^*)$. 
For the operators of that type, under some additional assumptions, we can show 
that we recover the class of operators with compact resolvent, as one would like
to have for genuine Laplace operators. 
We will not dwell here on this point making the conditions precise, but just sketch the argumentation.
First, consider the operator 
$$T = \left( P_{11} (\delta_1)^2  + P_{12} \delta_1 \delta_2 
   + P_{22} (\delta_2)^2  + \lambda \right) (1 +\triangle)^{-1}, $$
where $\triangle$ is the ,,flat'' Laplace operator on the 
noncommutative torus:
$$ \triangle = (\delta_1)^2  (\delta_2)^2, $$
and $P_{ij}$ are elements from $JA_\theta J^*$ (or the matrix algebra over it).
Note that each of the components of the sum is bounded, hence $P$ is 
bounded. Further, if the operator $P=P_{ij} \delta_i \delta_j$ is assumed
to be elliptic, then using similar arguments as in \cite{CoNo}

one can 
show that for some $\lambda$, $P+\lambda$ has no kernel 
and therefore $T$ is invertible. Hence, one can see that $(P+\lambda)^{-1}$ is compact.\\
We postpone to future work 
the detailed analysis (which would necessarily make use of noncommutative Sobolev spaces, 
see \cite{Ro08} and \cite{SaSt} for first examples).   \\
 
In addition to that property (which certainly holds only for some class of
the operators of the type considered) we can establish few more algebraic
properties. 
The first one regards the first order differential calculi.
\begin{lemma}
For any $D$ the bimodule of one forms $\Omega^1_D(A_\theta)$ is isomorphic to 
$A_\theta\oplus A_\theta$.
\end{lemma}
\begin{proof}
Oberve that the forms:
$$ \omega_i = (U_i)^* [D, U_i], $$
are central (that is $[a, \omega_i]=0$ for every $a \in A_\theta$) and
generate $\Omega^1(A_\theta)$ as a left (or right) module.  

We only need to check that that module is free. Assume that there exist $a_i \in A_\theta$
such that $ \sum_i a_i \omega_i = 0 $. This implies $ \sum_i e_j^i a_i = 0 $ for $j=1,2$. 
However, since we assumed that matrix $e_j^i$ was invertible, we immediately have $a_i=0$.
\end{proof}
The second one is the orientation property (we refer to \cite{Co95} for the terminology).
\begin{lemma}
Let $J A_\theta J^*\otimes A_\theta $ be a $A_\theta $-bimodule with
$a(JbJ^*\otimes c)d=JbJ^*\otimes acd$. For each $D$, there exists 
$c \in J A_\theta J^*\otimes A_\theta \otimes A_\theta \otimes A_\theta $,
\begin{equation}\label{orient}
c = \frac{1}{2i} \sum_{a,b,j,k} 
\epsilon_{ba} E^b_k E^a_j \otimes U_k^* U_j^* \otimes U_j \otimes U_k,
\end{equation}
where
$\epsilon_{ba} $ is the antisymmetric tensor and $E^b_k$ is 
the inverse of the matrix $e^k_j$, i.e. $E^b_k e^k_j = \delta^b_j$, 
such that 
$$\pi_D(c):= 
\frac{1}{2i}\sum_{a,b,j,k,p,r} \epsilon_{ba} E^b_k E^a_jU_k^* U_j^* [D,U_j][D,U_k] = \chi \ ,$$
where $\chi$ is  the grading operator on $\CH$.
If we view $c$ as a Hochschild 2-chain  with values in a  $A_\theta $-bimodule 
$J A_\theta J^*\otimes A_\theta $,
where $a(JbJ^*\otimes c)d=JbJ^*\otimes bcd$,
$c$ is a  cycle if 
\begin{equation}\label{zwei}
[E^1_1,E^2_2]=0 \quad\hbox{and}\quad  [E_1^2, E_2^1]=0.
\end{equation}
\end{lemma}
\begin{proof}
A straightforward calculation shows that
$$
\pi_D(c) = 
        \frac{1}{2i}\sum_{a,b,j,k,p,r} \epsilon_{ba} E^b_k E^a_j e_p^j e_r^k \sigma^p \sigma^r \\
       = \frac{1}{2i} \sum_{p,r} \epsilon_{pr} \sigma^p \sigma^r = \sigma^3 =\chi.
$$
The computation that the Hochschild boundary of $c$  vanishes is straightforward and 
leads to the conditions (\ref{zwei}) of the lemma. 
\end{proof}

This property is especially interesting since it is not known for the so called {\em gauge} perturbations of the Dirac operator. It is worth to mention that, to the best of our knowledge, it 
is the first instance (apart from the case of $0$-dimensional spectral triples over finite algebras) where the coefficients of the  Hochschild cycle have to be taken in a larger 
$A_\theta $-bimodule $J A_\theta J^*\otimes A_\theta $ rather than in $ A_\theta $ itself,
which usually suffices.

Note that the terms valued in $J A_\theta J^*$ are "spectators" from the point of view
of the Hochschild boundary operator and although we obtain a nontrivial condition for
the commutation relations between the elements of the coframe $E_i^j$ still it does not 
restrict it to a completely commutative case. On the other hand, one might think of relaxing the Hochschild boundary condition by requiring, 
for instance, that $\Phi(bc)=0$, where $\Phi = \phi \otimes \id \otimes \id$ and
$\phi$ is an arbitrary character on $J A_\theta J^*$. 

It is also interesting that the conditions (\ref{zwei}) for $c$ to be a cycle would automatically
enforce that $\det E$ and hence $\sqrt{g}$ can be unambigously defined, 
so it is then possible to detach $\sqrt{g}$ from the product $\sqrt{|g|}\, R$,
and define scalar curvature alone. We leave that issue for further investigations. 

Finally, let us comment the link with other operators. We note that the Dirac operator, which we propose is a natural generalization of the operator 
\begin{equation}
\sigma^1 \delta_1 + J \rho J^* \sigma^2 \delta_2, 
\label{dirtorr}
\end{equation}
where $\rho$ is a positive element of $A_\theta$. The class of such 
operators arose from construction of Dirac operators compatible with 
connections on noncommutative principal $U(1)$ bundles \cite{DS,DSZ}.
We also note that with $e_j^k$ such that:
$$ e^1_1 = 1, \;\; e^1_2  = 0, \;\; e^2_2 = \Im(\tau), \;\; e^2_1 = \Re(\tau), $$
we recover $D$ to be just the Dirac operator of the flat metric in the conformal 
class $\tau$. 

Moreover, the family \eqref{dirtor} include the operators 
given by eq. (44) in the paper \cite{conmos2} (which however studies 
the  modular, or twisted, spectral triples given by eq (45) therein).\\
Note that in all those three instances the condition (\ref{zwei}) is satisfied.

\section{The Gauss-Bonnet theorem and curvature on the noncommutative torus}

\subsection{Symbol calculus on Noncommutative Torus}

The symbol calculus defined in \cite{CoTr} and developed further
in \cite{conmos2} is easily generalized to the case of the operators
defined above.

Let us define a differential operator of order at most $n$ as:

$$ P_n = \sum_{i=1,2} \sum_{0 \leq j \leq n} \sum_{0 \leq k \leq j}
a_{ij} \delta_1^{k} \delta_2^{j-k}, $$
where $a_{ij}$ are in 
$J A_\theta J^{-1}$.

Let $\rho$ be $C^\infty$ function from $\R^2$ to $J\CA_\theta J^*$,
which is homogeneous of order $n$, satisfying certain bounds \cite{CoTr}.
For every symbol $\rho$ we define the operator on $\CH_\theta$:
$$ P_\rho(a) = \frac{1}{(2\pi)^n} \int e^{-i s \xi} \rho(\xi) \alpha_s(a) ds dx.$$
where 

$$ \alpha_s(U^\alpha) = e^{i s \cdot \alpha} U^\alpha. $$

The Dirac operator (\ref{dirtor}) could certainly be expressed in this
way and its symbol reads:

\begin{equation}
P(D) = \sum_{j,\mu=1}^2 \left( \sigma^j e_{j}^\mu \xi_\mu 
     + \frac{1}{2} \sigma^j \delta_\mu(e_{j}^\mu) \right) .
\label{dirsym}
\end{equation}
\subsection{The computations}

As shown in \cite{CoTr} by pseudodifferential calculations the value 
$\zeta(0)$ at the origin of the zeta function of the operator $D^2$ is 
given by 
$$ 
\zeta(0)= - \int \mathfrak{t} (b_2(\xi)) \, d \xi,
$$
where $b_2(\xi)$ is a symbol of order $-4$ of the pseudodifferential 
operator $(D^2+1)^{-1}$. It can be computed by recursion from the 
symbol $a_2(\xi) + a_1(\xi) + a_0(\xi)$ of $D^2$ as follows:
\begin{eqnarray}
b_2&=& -(b_0 a_0 b_0 + b_1 a_1 b_0 + \partial_1(b_0) \delta_1(a_1) b_0 +
\partial_2(b_0)\delta_2(a_1)b_0 \nonumber \\
&& + \partial_1(b_1)\delta_1(a_2)b_0 + \partial_2(b_1) \delta_2(a_2)b_0 + \oh\partial_{11}(b_0)\delta_1^2(a_2)b_0  \nonumber \\
&& \oh\partial_{22}(b_0)\delta_2^2(a_2)b_0  + \partial_{12}(b_0)\delta_{12}(a_2)b_0), 
\nonumber
\end{eqnarray}
where 
\begin{eqnarray}
b_1&=& -(b_0 a_1 b_0  + \partial_1(b_0) \delta_1(a_2) b_0 +
\partial_2(b_0)\delta_2(a_2) b_0), \nonumber
\end{eqnarray}
$$ b_0=(a_2+1)^{-1}$$
and 
$$
\partial_1  =\frac{\partial}{\partial \xi_1},
             \quad \quad
             \partial_2=\frac{\partial}{\partial \xi_2}.
$$
In or case we have
$$  
a_2(\xi)= \sigma^j \sigma^k  e_j^\mu e_k^\nu \xi_\mu \xi_\nu \ ,
$$
$$
a_1(\xi)=\sigma^j \sigma^k 
\{ \oh e_j^\mu \delta_\nu( e_k^\nu)  + e_j^\nu \delta_\nu( e_k^\mu) +
\oh \delta_\nu( e_j^\nu) e_k^\mu \} \xi_\mu\ ,
$$
and
$$
a_0(\xi)= \sigma^j \sigma^k 
\{ \ohh \delta_\mu( e_j^\mu)\delta_\nu( e_k^\nu) + 
\oh e_j^\mu \delta_\nu( e_k^\nu) \}\ .
$$

Due to the noncommutativity the computation of several hundreds of terms occuring 
in the formula for $\zeta(0)$ is a formidable task, which needs a symbolic calculation 
assistence and ingeneous manipulations. Even so, it is hardly possible to obtain
a final result without any simplifying assumption. That was indeed the case of 
conformal-type metric in \cite{CoTr}. 

Instead, we propose to perform perturbative analysis of the resulting terms having
assumed that the metric is a slight deviation of the standard equivariant one.

We assume that the components of two-frame can be expanded as 
$$
e_j^\mu = \delta_j^\mu + \varepsilon h_j^\mu ,
$$

where $\delta_j^\mu$ is the Kronecker delta and $ h_j^\mu \in JA_\theta J^*$. The initial 
unperturbed metric is thus flat, with $\tau=i$. This can be assumed without any loss
of generality, as we shall explain a posteriori that our result is valid for arbitrary 
initial constant metric and thus for arbitrary conformal class $\tau$. 

For the assumed form of the Dirac-type operator we shall compute the volume 
form as well as the curvature functional, which in the case of taking its value
on the unit of the algebra shall give us the (perturbative expansion) of the
Gauss-Bonnet theorem.

\subsection{Volume functional and absolute continuity}
 
Let us consider the following functional on the algebra 
$JA_\theta J^*$:
$$ JA_\theta J^* \ni a \mapsto \hbox{Res}_{s=2} \zeta_a(s), $$
where
$$ \zeta_a(s) = \hbox{Tr}\ a |D|^{-s}. $$
Using the symbolic calculus we can relate the residue \cite{fatkha2,fatwon} to the following expression:
$$ \hbox{Res}_{s=2} \zeta_a(s) = \int_{S^1} \mathfrak{t} (a b_0(\xi)) d\Omega(\xi). $$

In the perturbative expansion one then obtain the relevant
volume  functionals (as for $a=1$ one obtains the volume 
as the leaing term of the spectral action expansion):

\begin{equation}
\begin{aligned}
\hbox{Res}_{s=2} & \zeta_a(s) \sim  \\
&(2\pi)^2 \mathfrak{t} \left( a \left( 1 - \varepsilon(h_1^1 + h_2^2) 
+ \varepsilon^2 \left( (h_1^1)^2 + (h_2^2)^2 + \frac{1}{2} \left[ h_1^2, h_2^1 \right]_+ 
 + \frac{1}{2} \left[ h_1^1, h_2^2\right]_+ \right) \right) \right) + O(\varepsilon^3),
\label{vf1}
\end{aligned}
\end{equation}

Note that the functional $\hbox{Res}_{s=2} \zeta_b(s)$ taken for $b \in A_\theta$ 
factorizes, i.e.:

\begin{equation}
\begin{aligned}
 \hbox{Res}_{s=2} & \zeta_b(s) \sim \\
 & (2\pi)^2 \mathfrak{t}(b)  
\mathfrak{t}\left( \left( 1 - \varepsilon(h_1^1 + h_2^2) 
+ \varepsilon^2 \left( (h_1^1)^2 + (h_2^2)^2 + \frac{1}{2} \left[ h_1^2, h_2^1\right]_+ 
 + \frac{1}{2} \left[ h_1^1, h_2^2\right]_+ \right) \right) \right) + O(\varepsilon^3),
 \label{vf2}
\end{aligned}
\end{equation}

For this reason one might wonder whether the absolute continuity condition (see axiom 5 in \cite{Co95}) for the spectral triple makes sense in an unchanged form.

In particular it might be asked whether the
hermitian structure on the trivial module of spinors (which after 
completion gives the spinor space) should be valued in the algebra
$A_\theta$ or rather in the opposite algebra $J A_\theta J*$. 

Although still the formula provides a reasonable answer, one sees 
that the notion of volume form and volume measure (as computed
perturbatively in (\ref{vf1}) and (\ref{vf2})) becomes slightly different than in the classical case or in the case of flat Dirac operator. \\

\subsection{Scalar curvature} \label{curv}

The idea of the definition of scalar curvature appeared first in \cite{conmar}, expressed 
in terms of the second term of the heat expansion. For our purposes, as the symbols of 
the Dirac operator and its square are not in the algebra but in its commutant, we need 
to propose a slightly modified definition. 

We have also to take into account,
that we consquently use the state $ \mathfrak{t}$ corresponding to the 'flat' measure
 over $\T_\theta$.
Thus we search for the analogy not of the classical scalar curvature,
but rather of its product with $\sqrt{|g|}$. 
More precisely, we search for  the unique element in $\tilde R\in JA_\theta J^*$ such that:
$$  \textnormal{Trace}\, (a (D^2)^{-s/2})_{|_{s=0}}  
  = \frac{1}{12\pi} \mathfrak{t} (a \tilde R), \qquad \forall a \in JA_\theta J^* .
$$

The refinement of this notion (which classically vanishes)
is the unique  element $\tilde  R^\gamma \in JA_\theta J^*$ that satisfies:
$$  \textnormal{Trace}\, (\gamma a (D^2)^{-s/2})_{|_{s=0}} 
  =  \frac{1}{12\pi} \mathfrak{t}(a \tilde R^\gamma), \qquad \forall a \in JA_\theta J^*.
$$

Observe that these notions certainly make perfect sense, however, one 
might ask why we do not take into account similar functionals with $a \in A_\theta$ 
instead. The reason is the fact that the trace factorizes on products of elements from 
$A_\theta$ by the elements from the  $JA_\theta J^*$, i.e. 
for $a \in A_\theta$ and $b \in JA_\theta J^*$:
$$  \mathfrak{t}(a b) =  \mathfrak{t}(a)  \mathfrak{t}(b). $$ 

Using the calculus of symbols extended to $A_\theta$ and $JA_\theta J^*$ one can 
find that that the curvature functional would then vanish for all $a$ 
if it vanishes for $a=1$ (which is the Gauss-Bonnet theorem). We conjecture that 
this is the case for all admissible Dirac operators on the noncommutative 2-torus.
Of course, such situation is very particular for the case considered and might not 
be generic.

Below we present the results for the curvature and chiral curvature, seen as
elements of $JA_\theta J^*$ in this sense. In fact, even in this simple case the 
computations are much involved, as the expression for $b_2(\xi)$ alone counts 7100 
terms (up to $\varepsilon^2$) and its printout takes more than 100 pages. For this 
reason we present here only the final result, after integration with respect to $\xi$.\\

$$
\begin{aligned}
\tilde R =& 
2 \varepsilon \left( 
      + \delta_{1} \delta_{1} (h_2^2)  
      + \delta_{2} \delta_{2} (h_1^1)
      - \delta_{1} \delta_{2} (h_1^2)  
      - \delta_{2} \delta_{1} (h_2^1) \right) \\
  +& 
\varepsilon^2 \left(
     \left[ h_1^1, \delta_1\delta_2(h_2^1) +(\delta_1)^2(h_2^2) -2  (\delta_2)^2(h_1^1) \right]_+ 
   + \left[ h_2^2, \delta_1\delta_2(h_1^2) +(\delta_2)^2(h_1^1) - 2  (\delta_1)^2(h_2^2) \right]_+ \right. \\
  & \phantom{\frac{1}{12} \pi \varepsilon^2 }
     +\left[ h_1^2, 2 \delta_1\delta_2(h_2^2) + \delta_1\delta_2(h_1^1)
                   -(\delta_2)^2(h_1^2) - (\delta_1)^2(h_2^1) - (\delta_2)^2(h_1^2)  \right]_+ \\
  & \phantom{\frac{1}{12} \pi \varepsilon^2 }
     +\left[ h_2^1, 2\delta_1\delta_2(h_1^1) + \delta_1\delta_2(h_2^2)
                  -(\delta_2)^2(h_1^2)  - (\delta_1)^2(h_1^2) - (\delta_1)^2(h_2^1) \right]_+ \\
  & \phantom{\frac{1}{12} \pi \varepsilon^2 }
    +\left[ \delta_2(h_1^1) , 2\delta_1(h_2^1) + \delta_1(h_1^2) \right]_+ 
    + \left[ \delta_1(h_2^2) , 2\delta_1(h_1^2) + \delta_2(h_1^2) \right]_+ \\
  & \phantom{\frac{1}{12} \pi \varepsilon^2 }
    +\left[ \delta_1(h_1^1) , \delta_1(h_2^2) + \delta_2(h_2^1) \right]_+ 
    + \left[ \delta_2(h_2^2) , \delta_2(h_1^1) + \delta_1(h_1^2) \right]_+ \\
  & \phantom{\frac{1}{12} \pi \varepsilon^2 }
    -2 \left[ \delta_2(h_1^2) , \delta_2(h_2^1) + \delta_1(h_2^1) \right]_+ \\
    & \phantom{\frac{1}{12} \pi \varepsilon^2 } \left.
    -2 \left( \delta_2(h_1^2) \right)^2  - 2 \left( \delta_1(h_2^1) \right)^2  
    -4 \left( \delta_2(h_1^1) \right)^2  -4 \left( \delta_1(h_2^2) \right)^2 \right)  + O(\varepsilon^3)
\end{aligned}
$$
where $[\cdot,\cdot]_+$ denotes anticommutator.\\

It is easy to see that the linear terms coincide with the ones from the
commutative case, whereas in the next order we obtain terms, which are not 
a straigthforward generalization of the commutative ones.

More interesting is the ,,chiral curvature'' $\tilde R^\gamma$. In our case, the
first order term vanishes (as expected), whereas the second term could
be expressed as a sum of commutators:
$$
\begin{aligned}
\tilde R^\gamma = & 
i \varepsilon^2 \left( 
\left[ h_1^1, \delta_2\delta_2(h_2^1 + h_1^2) 
-  \delta_1\delta_2(3 h_1^1 + 2 h_2^2) 
+ 3 \delta_1\delta_1(h_2^1) \right] \right. \\
& \phantom{xxxx} 
+ \left[ h_1^2, 3 \delta_1\delta_1(h_2^2) 
- 2 \delta_2\delta_2 (h_1^1) - 2 \delta_1\delta_2(h_1^2) 
+ \delta_1\delta_2 (h_2^1)  \right] \\
& \phantom{xxxx} 
+ \left[ h_2^1, 
2 \delta_1\delta_1(h_2^2) + 2 \delta_1\delta_2(h_2^1) 
-3\delta_2\delta_2(h_1^1) -   \delta_1\delta_2(h_1^2) \right] \\
& \phantom{xxxx} 
+\left[ h_2^2, \delta_1\delta_2(3 h_2^2 + 2 h_1^1) 
-  \delta_1\delta_1(h_1^2 + h_2^1) - 3 \delta_2\delta_2(h_1^2) \right] \\
& \phantom{xxxx} 
 + 3\left[ \delta_1(h_1^1), \delta_1(h_2^1) \right]
  + \left[ \delta_2(h_1^2), \delta_1(h_2^1) \right]
 -2 \left[ \delta_2(h_1^1), \delta_1(h_2^2) \right] 
 -3 \left[ \delta_2(h_1^2), \delta_2(h_1^1) \right] \\
& \phantom{xxxx} 
 +3\left[ \delta_2(h_1^2), \delta_2(h_2^2) \right]
 -3\left[ \delta_2(h_2^1), \delta_2(h_1^1) \right]
 - \left[ \delta_1(h_1^1), \delta_2(h_2^2) \right] 
 -2\left[ \delta_2(h_2^1), \delta_1(h_1^2) \right] \\
& \phantom{xxxx} \left.
  -3\left[ \delta_1(h_2^2), \delta_1(h_1^2) \right]
  +3\left[ \delta_1(h_2^1), \delta_1(h_2^2) \right]
 \right) + O(\varepsilon^3),
\end{aligned}
$$
so it vanishes in the limiting case of commutative torus.

\subsection{The perturbative Gauss-Bonnet theorem}

The statement of the Gauss-Bonnet theorem couldbe phrased either as vanishing
of the curvature functional at $a=1$ or, equivalently as:
$$ \mathfrak{t} (\tilde R) = 0. $$

Using the periodicity of the trace and the fact that trace is invariant with
respect to $U(1)$ actions (which translates to $ \mathfrak{t}(\delta_i(a))=0$ for
$i=1,2$ and any $a \in JA_\theta J^*$, we obtain that at order $\varepsilon$ the
statement holds trivially, all terms are derivations of some elements from $JA_\theta J^*$.

At the order $\varepsilon^2$ (we dentote the $\varepsilon^2$ part
of $\tilde R$ by $\tilde R_{\varepsilon^2}$) we need first to use the cyclicity 
of the trace, and then the Lebniz rule, to reduce (after three consecutive iterations) the expression to the following one:
$$
\begin{aligned}
\mathfrak{t}(\tilde R)_{\varepsilon^2} = 
 \mathfrak{t} \left( \right. &
  - 4 h_1^1 \left( \delta_1 \delta_2(h^1_2) - \delta_2 \delta_1 (h^1_2) \right)
   +h_1^1 \left( \delta_1 \delta_2(h^2_1) - \delta_2 \delta_1 (h^2_1) \right)  \\
& \left. -3 h_1^2 \left( \delta_1 \delta_2(h^2_2) - \delta_2 \delta_1 (h^2_2) \right)
  +  h_2^1 \left( \delta_2 \delta_1(h^2_2) - \delta_1 \delta_2 (h^2_2) \right) \right) \\  
=& \; 0,
\end{aligned}
$$
(the last equality follows since $\delta_1$ and $\delta_2$ commute).

We stress in the above computations we used the cyclicity and invariance 
of $\mathfrak{t}$ the properties of the derivations $\delta_j$ (like 
integration by parts rule), but never their explicit form. Therefore, the result 
we obtained holds for any linear combination (with constant coefficients) 
of $\delta_j$, in particular for 
$$
\delta_1= \delta_1,\quad \delta_2' = (\tau - \overline{\tau})/2i \delta_2,
$$
which corresponds to the complex (or equivalently conformal) class 
on $\T_\theta$ labeled by  $\tau$.

\section{Conclusions}
Although this project has still to be further developed, we have already 
encountered several new and interesting phenomena. First of all, we observe 
that the existence of {\em real} spectral triples is necessary to provide a kind of 
background geometry. Then one can consistently introduce a class of new 
Dirac operators, which in turn correspond to new {\em noncommutative metrics}, 
extending the so-far considered examples. 

Our aim in this paper was to outline the possibilities, which appear to be much
broader than originally believed. Still, many of the properties of the introduced
objects are to be studied more closely, which we shall adress in future work. 
We believe that the introduction of families of (curved) metrics 
might open a new interstin directions in the investigations of noncommutative
manifolds as metric spaces.



\begin{thebibliography}{99}
\bibitem{bhumar} T. A. Bhuyain, M.  Marcolli, 
{\it  The Ricci flow on noncommutative two-tori,} arXiv:1107.4788.
\bibitem{cohcon} P. B. Cohen,  A. Connes,  
{\it Conformal geometry of the irrational rotation algebra,}  Preprint MPI (92-93).
\bibitem{Co80} A. Connes,  
{\it $C^*$-alg\`ebres et g\'eom\'etrie diff\'erentielle}, 
C.R. Acad. Sc. Paris, t.~290, S\'erie A, 599-604, 1980.
\bibitem{CoNo}
A. Connes, {\it C* alg\'ebres et g\'eom\'etrie differentielle.}
C.R. Acad. Sci. Paris, Ser. A-B , 290, 1980.
\bibitem{Co94}
{A. Connes}, {\it Noncommutative Geometry}, Academic Press, 1994.
\bibitem{Co95}
{A. Connes},
{\it Gravity coupled with matter and foundation of non-commutative
geometry},Comm. Math. Phys. 182, 155--176, 1996.
\bibitem{conmar} A.~Connes, M.~Marcolli, 
\textit{Noncommutative geometry, quantum fields and
  motives}, Colloquium Publications, vol.~55, AMS, 2008. 
\bibitem{conmos2} A. Connes, H. Moscovici, {\it Modular curvature for noncommutative
 two-tori,} arXiv:1110.3500.
\bibitem{CoTr} A. Connes, P. Tretkoff, 
{\it The Gauss-Bonnet theorem for the noncommutative two torus},
arXiv:0910.0188.
\bibitem{DS}
L. D\k abrowski, A.Sitarz,
{\em Noncommutative circle bundles and new Dirac operators}
arXiv:1012.3055v2
\bibitem{DSZ}
L. D\k abrowski, A.Sitarz, A. Zucca,
{\em Noncommutative circle bundles over odd dimensional spaces},
in preparation
\bibitem{MCC}
D. Essouabri, B. Iochum, C. Levy, A. Sitarz,
{\em Spectral action on noncommutative torus}, 
J. Noncommut.Geom. {\bf 2}, 53--123, 2008. 
\bibitem{fatkha} F. Fathizadeh, M. Khalkhali, 
{\it The Gauss-Bonnet theorem for noncommutative
two tori with a general conformal structure}, arXiv:1005.4947.
\bibitem{fatkha1} F. Fathizadeh, M. Khalkhali, 
{\it Scalar curvature for the noncommutative two torus,}  arXiv:1110.3511.
\bibitem{fatkha2} F. Fathizadeh, M. Khalkhali, 
{\it Weyl's Law and Connes' Trace Theorem for 
Noncommutative Two Tori}, F.Fathizadeh, M.Khalkhali, arXiv:1111.1358,
\bibitem{fatwon} F. Fathizadeh, M. W. Wong, 
{\it Noncommutative residues for pseudo-differential operators on the noncommutative two-torus}, 
Journal of Pseudo-Differential Operators and Applications, {\bf 2}(3), 289--302, 2011.
\bibitem{PaSi}
M. Paschke, A. Sitarz, {\it On Spin Structures and Dirac Operators on the Noncommutative Torus}, 
Lett.Math.Phys 77 (3) 317-327, 2006.
\bibitem{Ro08}
J. Rosenberg {\it Noncommutative Variations on Laplace Equation},
Anal. PDE 1 (1), 95-114, 2008.
\bibitem{SaSt}
A. Yu. Savin, B. Yu. Sternin
{\it Noncommutative elliptic theory. Examples},
Proceedings of the Steklov Institute of Mathematics
Volume 271, Number 1, 193-211, 2010.
\end{thebibliography}
\end{document}